\documentclass[12pt, letterpaper]{article}
\usepackage[utf8]{inputenc}
\usepackage[english]{babel}
\usepackage{amsmath}
\usepackage{amsfonts}
\usepackage{amssymb}
\usepackage{amsthm}
\usepackage{graphicx}
\usepackage{tikz}
\usepackage{subcaption}

\usetikzlibrary{shapes.geometric}

\usepackage[left=1in,top=1in,right=1in,bottom=1in]{geometry}

\newtheorem{teo}{Theorem}
\newtheorem{lem}{Lemma}

\newtheorem{exa}{Example}

\newtheorem{cor}{Corollary}

\numberwithin{equation}{section}
\title{Crossings in Randomly Embedded Graphs}
\author{Santiago Arenas-Velilla \\\small Centro de Investigaci\'on en Matem\'aticas,\\[-0.8ex]
\small Guanajuato, Gto. 36000, Mexico\\
\small\tt santiago.arenas@cimat.mx \\\and Octavio Arizmendi\\ \small Centro de Investigaci\'on en Matem\'aticas,\\[-0.8ex]
\small  Guanajuato, Gto. 36000, Mexico\\
\small\tt octavius@cimat.mx\\}
\begin{document}
\maketitle

\begin{abstract}
    {We consider the number of crossings in a graph which is embedded randomly on a convex set of points. We give an estimate to the normal distribution in Kolmogorov distance which implies a convergence rate of order $n^{-1/2}$ for various families of graphs, including random chord diagrams or full cycles.}
\end{abstract}

\textbf{keywords}: crossings, random graphs, normal approximation.
\section{Introduction}
Let $G = (V,E)$ be a graph with vertex set $V$ and edge set $E \subset V \times V$, which is embedded randomly in a convex set of points. We are interested in the random variable counting the number of crossings under this embedding.  

Formally, for a graph $G = (V,E)$ with vertex set $[n]=\{1,\ldots, n\}$, an embedding given by the permutation $\pi:[n]\to[n]$, is the graph isomorphism induced by the permutation $\pi$. The crossings of such embedding  are given by the set $\{(a,b,c,d)|\{a,b\}, \{c,d\} \in E, \pi(a)>\pi(c) >\pi(b) > \pi(d)\}$. Figure \ref{pathfig_example} shows graphical representation of a couple embeddings of a path graph $P_{20}$. The first one having $40$ crossing and the second one having $60$ crossings. 
 
\begin{figure}[ht]
\centering
\begin{subfigure}[b]{0.45\linewidth}
\includegraphics[width=\linewidth,height=200pt]{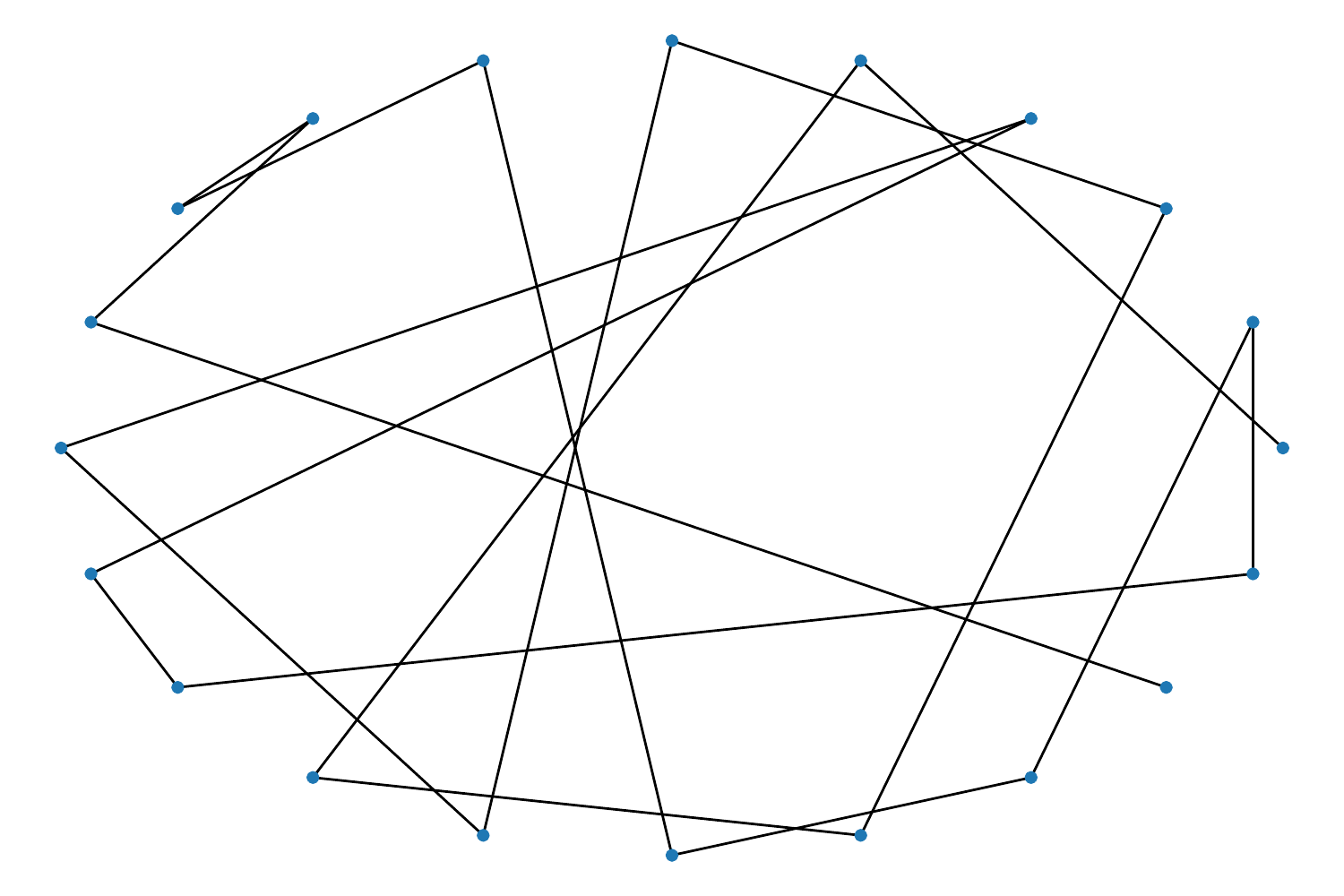}
\caption{$P_{20}$ with 40 crossings}
\label{P_297}
\end{subfigure}
\begin{subfigure}[b]{0.45\linewidth}
\includegraphics[width=\linewidth,height=200pt]{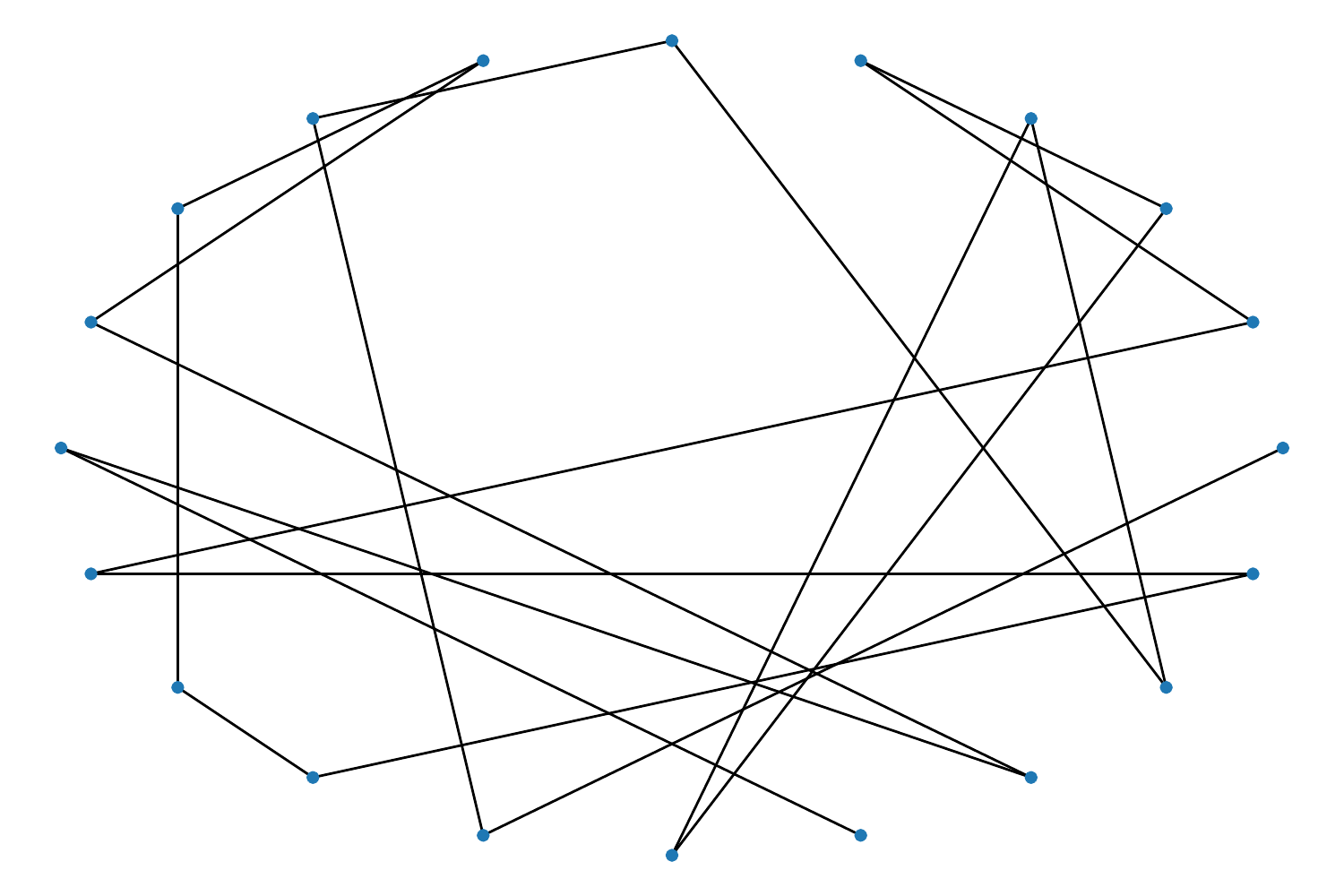}
\caption{$P_{20}$ with 60 crossings}
\label{P_383}
\end{subfigure}
\caption{Examples of an embedding of a path with 20 vertices}
\label{pathfig_example}
\end{figure}

To our best of our knowledge there is not much work about general graphs. The paper by Flajolet and Noy \cite{FN2000} considers the case where $G$ is a union of disjoint edges (is called a matching, a pairing or a chord diagram) and proves a central limit theorem. This result is also proved with the use of weighted dependency graphs, in \cite{Feray2018}. More important to us the recent paper by Paguyo \cite{paguyo2021convergence} gives a rate of convergence in that case. Another related paper is \cite{ACH}, where the authors consider a uniform random tree.

In this paper, we will show that under some asymptotic behaviour of very precise combinatorial quantities of the graph, the random variable counting the number of crossings in a random embedding approximates a normal distribution with mean $\mu$ and variance $\sigma^2$ which can  be calculated precisely (see Lemmas \ref{lem:mean} and \ref{lem:2ndmoment}). Moreover, we give a convergence rate in this limit theorem.

\begin{teo}
Let $G$ be a graph with maximum degree $\Delta$, and let $X$ be the number of crossings of a random uniform embedding of $G$. Let $\mu$ and $\sigma^2$ the mean and the variance of $X$. Then, with  $W = (X-\mu) /\sigma$,  
\begin{equation}
d_{Kol}(W,Z) \leq \frac{4 m_2(G) \Delta m}{3 \sigma^2}\left(\frac{6 \Delta m}{\sigma} + \sqrt{1- \frac{6m_4(G)}{m_2(G)^2} + \frac{m(\Delta-1)^2}{m_2(G)} } \right),
\label{kolmogorov_distance}
\end{equation}
where $m$ is the number of edges of $G$, $m_r(G)$ is the number of $r$-matching of $G$, and $Z$ is a standard Gaussian random variable.
\end{teo}

Examples of families of graphs that satisfy such a normal approximation, with a rate proportional to $1/\sqrt{n}$ are pairings, cycle graphs, path graphs, union of triangles, among others. We explain in detail these examples in Section 4.

We should mention that our method of proof resembles the one used by Paguyo in \cite{paguyo2021convergence}, for the case of pairings. The main idea is to write the number of crossing as a sum of indicators variables and then consider the size biased transform in the case of sums of Bernoulli variables. However, there is a crucial difference between Paguyo's way to write such variables and how we do it, which in our opinion is more flexible. To be precise, Paguyo considers for each four points $a<b<c<d$ in the circle the indicator that there is a crossing formed by the edges $\{a,c\}$ and $\{b,d\}$. This random variable is easy to handle for the case of a pairing but for a general given graph, even calculating the probability of such indicator to be $1$ can be very complicated. Our approach instead looks at a given $2$-matching in the graph $G$ and consider the indicator random variable of this $2$-matching, when embedded randomly, to form a crossing.

\section{Preliminaries}
In this section we establish some notations for graphs and remind the reader about the main tool that we will use to quantify convergence to a normal distribution: the size bias transform.

\subsection{Notation and Definitions on Graphs}

A graph is pair $G=(V,E)$, where $E\subset\{\{v,w\}|v,w\in V\}$. Elements in $V$ are called vertices and elements in $E$ are called edges. An edge $\{v,w\}$ is sometimes written as $v \sim_G w $ or $v \sim w$ if the underlying graph $G$ is clear. The number or vertices $|V|$, will be denoted by $n$, while the number of edges, $|E|$, will be denoted be $m$.

For a vertex $v$, we say that $w$ is a neighbour of $v$, if $\{v,w\}\in E$. The number of neighbours of $v$ is called the degree of $v$, denoted by $deg(v)$. The largest degree among all vertices in a graph will be denoted by $\Delta$.

A subgraph of $G$, is a graph, $H=(W,F)$, such that $W\subset V$ and $F\subset G$. We say that $G = (V,E)$ and $G' = (V',E')$ are \textit{isomorphic} if there exist a bijection $\varphi : V \to V'$ such that $\{ u , v \} \in E $ if and only if $\{ \varphi(u) , \varphi(v) \} \in E'$, for all $u,v \in V$. 

An \textit{$r$-matching} in a graph $G$ is a set of $r$ edges in $G$, no two of which have a vertex in common. We denote by $M_r(G)$ the set $r$-matchings of $G$ and by $m_r(G)$ their cardinality. Note the $m_1=m$ corresponds to the number of edges of the graph $G$. 

\subsection{Size Bias Transform}

Let $X$ be a positive random variable with mean $\mu$ finite. We say that the random variable $X^s$ has the \textit{size bias} distribution with respect to $X$ if for all $f$ such that $\mathbb{E}[Xf(X)] < \infty$, we have 
$$\mathbb{E}[Xf(X)] = \mu \mathbb{E}[f(X^s)].$$

In the case of $X = \sum_{i=1}^{n} X_i$, with $X_i$'s positive random variables  with finite mean $\mu_i$, there is a recipe to construct $X^s$ (Proposition 3.21 from \cite{ross2011fundamentals}) from the individual size bias distributions of the summands $X_i$:

\begin{enumerate}
\item For each $i = 1, \ldots, n$, let $X_i^{s}$  having the size bias distribution with respect to $X_i$, independent of the vector $(X_j)_{j \neq i}$ and $(X_j^s)_{j\neq i}$. Given $X_i^s = x$, define the vector $(X_{j}^{(i)})_{j \neq i}$ to have the distribution of $(X_j)_{j\leq i}$ conditional to $X_i = x$.
\item Choose a random index $I$ with $\mathbb{P}(I = i) = \mu_i / \mu$, where  $\mu = \sum \mu_i$, independent of all else. 
\item Define $X^s = \sum_{j \neq I} X_{j}^{(I)} + X_I^s$.
\end{enumerate}
It is important to notice that the random variables are not necessarily independent or have the same distribution. Also, $X$ can be an infinite sum (See Proposition 2.2 from \cite{chen2011normal}). 

If $X$ is a Bernoulli random variable, we have that $X^s = 1$. Indeed, if $\mathbb{P}(X=1)=p$, $\mathbb{E}(X)=p=\mu$ and then $$\mathbb{E}[Xf(X)] =(1-p)(0f(0))+p(1f(1))=pf(1)=\mu f(1)=\mu \mathbb{E}[f(1)].$$
Therefore, we have the following corollary (Corollary 3.24 from \cite{ross2011fundamentals}) by specializing the above recipe. 
\begin{cor}
\label{corollary_sizebiasBernoulli}
Let $X_1, X_2, \ldots, X_n$ be Bernoulli random variables with parameter  $p_i$. For each $i = 1, \ldots, n$ let $(X_j^{(i)})_{j \neq i}$ having the distribution of $(X_j)_{j \neq i}$ conditional on $X_i =1$. If $X = \sum_{i=1}^n X_i$, $\mu = \mathbb{E}[X]$, and $I$ is chosen independent of all else with $\mathbb{P}(I= i) = p_i /\mu$, then $X^s = 1+\sum_{j \neq I} X_j^{(I)}$ has the size bias distribution of $X$. 
\end{cor}

The following result (Theorem 5.3 from \cite{chen2011normal}) gives us bounds for the Kolmogorov distance, in the case that a bounded size bias coupling exists. This distance is given by
$$d_{Kol}(X,Y) := \sup_{z \in \mathbb{R}} |F_X(z) - F_Y(z)|,$$
where $F_X$ and $F_Y$ are the distribution functions of the random variables $X$ and $Y$.
\begin{teo}\label{teoKolmo}
Let $X$ be a non negative random variable with finite mean $\mu$ and finite, positive variance $\sigma^2$, and suppose $X^s$, have the size bias distribution of $X$, may be coupled to $X$ so that $|X^s-X| \leq A$, for some $A$. Then with $W = (X- \mu)/ \sigma$,

\begin{equation}\label{Kolmogorov}
d_{Kol}(W,Z) \leq   \frac{6 \mu A^2}{\sigma^3} +\frac{2 \mu \Psi}{\sigma^2},
\end{equation}
where $Z$ is a standard Gaussian random variable, and $\Psi$ is given by 
\begin{equation}\label{Psi}
\Psi = \sqrt{\mathrm{Var}(\mathbb{E}[X^s-X | X])}
\end{equation}

\end{teo}

\section{Mean and Variance}

Let $S_n$ be the set of permutation of $n$ elements. For a permutation $\pi \in S_n$, and a graph $G$, let $G_\pi$ be the graph whose edges are given by 
$$v \sim _{G} w \Longleftrightarrow \pi(v) \sim_{G_\pi} \pi(w), \qquad \forall v,w \in V.$$
For a random uniform permutation $\pi$, let $X := X(G_\pi)$ be the random variable that counts the number of crossings of $G_\pi$, that is 
\begin{equation} X= \sum_{i \in M_2(G)} \mathbb{I}_{\{i \text{ is a crossing in } G_\pi \}} = \sum_{i \in M_2(G)} Y_i\end{equation}
where $M_2(G)$ is the set of 2-matching of $G$. 

In this section we give a formula for the mean and variance of the random variable $X$ in terms of the number of subgraphs of certain type.

\begin{lem}\label{lem:mean}For a graph $G$, if $X$ denote the number of crossings in a random embedding on a set of $n$ points in convex position, then its expectation is given by
$$\mu := \mu(G) =\mathbb{E}(X)=\frac{1}{3}m_2(G),$$
where $m_2(G)$ denotes the number of 2-matching of $G$.
\end{lem}

\begin{proof}
For each $i\in M_2(G)$, we notice that $Y_i \sim \mathrm{Bernoulli}(1/3)$. Indeed, if $i$ consist of two edges $\{v_1,v_2\}$  and $\{v_3,v_4\}$, the probability to obtain a crossing only depends on the cyclic orders in which $v_1,v_2,v_3$ and $v_4$ are embedded in $\{1,\dots,n\}$, not in the precise position of them. From the $6$ possible orders, only $1/3$ of them yield a crossing. See Figure \ref{six_orders}, for the $6$ possible cyclic orders of $v_1,v_2,v_3$ and $v_4$. 

\begin{figure}[ht]
\centering 
\begin{tikzpicture}[scale = 0.9]

\draw (0,0) node[black][left] {$v_1$} -- (1,0) ;
\filldraw (0,0) circle (1.5pt);
\filldraw (1,0) node[black][right]{$v_2$} circle (1.5pt);
\draw (0,-1) node[black][left] {$v_4$} -- (1,-1) ;
\filldraw (0,-1) circle (1.5pt);
\filldraw (1,-1) node[black][right]{$v_3$} circle (1.5pt);

\draw (3,0) node[black][left] {$v_1$} -- (4,0) ;
\filldraw (3,0) circle (1.5pt);
\filldraw (4,0) node[black][right]{$v_2$} circle (1.5pt);
\draw (3,-1) node[black][left] {$v_3$} -- (4,-1) ;
\filldraw (3,-1) circle (1.5pt);
\filldraw (4,-1) node[black][right]{$v_4$} circle (1.5pt);

\draw (6,0) node[black][left] {$v_1$} -- (7,-1) ;
\filldraw (6,0) circle (1.5pt);
\filldraw (7,-1) node[black][right]{$v_2$} circle (1.5pt);
\draw (6,-1) node[black][left] {$v_4$} -- (7,0) ;
\filldraw (6,-1) circle (1.5pt);
\filldraw (7,0) node[black][right]{$v_3$} circle (1.5pt);

\draw (9,0) node[black][left] {$v_1$} -- (10,-1) ;
\filldraw (9,0) circle (1.5pt);
\filldraw (10,-1) node[black][right]{$v_2$} circle (1.5pt);
\draw (9,-1) node[black][left] {$v_3$} -- (10,0) ;
\filldraw (9,-1) circle (1.5pt);
\filldraw (10,0) node[black][right]{$v_4$} circle (1.5pt);

\draw (12,0) node[black][left] {$v_1$} -- (12,-1) ;
\filldraw (12,0) circle (1.5pt);
\filldraw (12,-1) node[black][left]{$v_2$} circle (1.5pt);
\draw (13,0) node[black][right] {$v_3$} -- (13,-1) ;
\filldraw (13,0) circle (1.5pt);
\filldraw (13,-1) node[black][right]{$v_4$} circle (1.5pt);

\draw (15,0) node[black][left] {$v_1$} -- (15,-1) ;
\filldraw (15,0) circle (1.5pt);
\filldraw (15,-1) node[black][left]{$v_2$} circle (1.5pt);
\draw (16,0) node[black][right] {$v_4$} -- (16,-1) ;
\filldraw (16,0) circle (1.5pt);
\filldraw (16,-1) node[black][right]{$v_3$} circle (1.5pt);

\end{tikzpicture}
\caption{Possibles cyclic orders for a $2$-matching.}
\label{six_orders}

\end{figure}
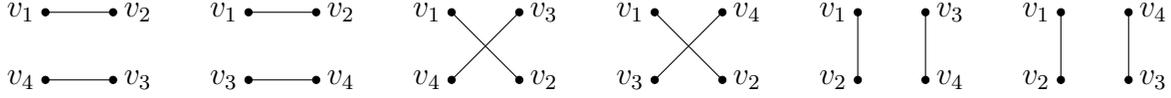

Summing over all $j$, the expected value of $X$ is 
$$\mathbb{E}X = \sum_{i \in M_2(G)} \mathbb{P}(i \text{ is a crossing}) = \frac{1}{3} m_2(G),$$
as desired.
\end{proof}

For the the second moment it is necessary  to expand $X^2$ to get
$$\mathbb{E}X^2 = \mathbb{E} \sum_{i,j \in M_2(G)} \mathbb{I}_{\{ i,j \text{ are crossings}\}} = \sum_{i,j \in M_2(G)} \mathbb{P}(i,j \text{ are crossings}).$$
The analysis for $\mathbb{E}X^2$  is significantly more complicated, since $\mathbb{P}(i,j \text{ are crossings})$ depends of how many edges and vertices the two $2$-matchings, $i$ and $j$, share.  Thus
the  previous sum can be divided in $8$ types, depending of how the $2$-matchings, $i$ and $j$, share edges and vertices as is shown in Figure \ref{double_sum}. We call such different configuration the ``kind of pair of  $2$-matching" .

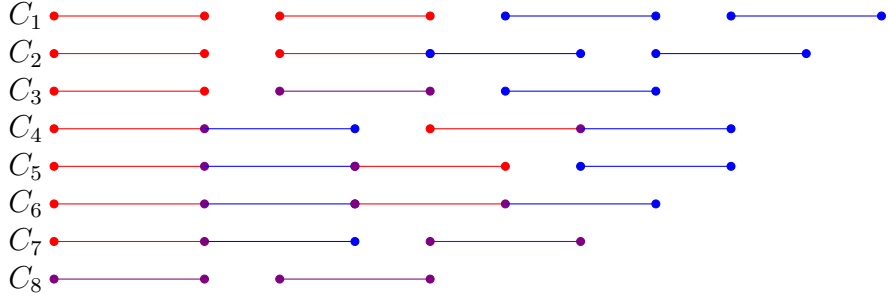
\begin{figure}[ht]
\centering 
\begin{tikzpicture}[scale = 1]
%case 1
\draw [red] (0,0) node[black][left] {$C_1$} -- (2,0) ;
\filldraw [red] (0,0) circle (1.5pt);
\filldraw [red] (2,0) circle (1.5pt);
\draw [red] (3,0) -- (5,0);
\filldraw [red] (3,0) circle (1.5pt);
\filldraw [red] (5,0) circle (1.5pt);
\draw [blue](6,0) -- (8,0);
\filldraw [blue] (6,0) circle (1.5pt);
\filldraw [blue] (8,0) circle (1.5pt);
\draw [blue](9,0) -- (11,0);
\filldraw [blue] (9,0) circle (1.5pt);
\filldraw [blue] (11,0) circle (1.5pt);

\draw [red] (0,-0.5) node[black][left] {$C_2$} -- (2,-0.5);
\filldraw [red] (0,-0.5) circle (1.5pt);
\filldraw [red] (2,-0.5) circle (1.5pt);
\draw [red] (3,-0.5) -- (5,-0.5);
\filldraw [red] (3,-0.5) circle (1.5pt);
\filldraw [red] (5,-0.5) circle (1.5pt);
\draw [blue](5,-0.5) -- (7,-0.5);
\filldraw [blue] (5,-0.5) circle (1.5pt);
\filldraw [blue] (7,-0.5) circle (1.5pt);
\draw [blue](8,-0.5) -- (10,-0.5);
\filldraw [blue] (8,-0.5) circle (1.5pt);
\filldraw [blue] (10,-0.5) circle (1.5pt);

\draw [red] (0,-1) node[black][left] {$C_3$} -- (2,-1);
\filldraw [red] (0,-1) circle (1.5pt);
\filldraw [red] (2,-1) circle (1.5pt);
\draw [violet] (3,-1) -- (5,-1);
\filldraw [violet] (3,-1) circle (1.5pt);
\filldraw [violet] (5,-1) circle (1.5pt);
\draw [blue](6,-1) -- (8,-1);
\filldraw [blue] (6,-1) circle (1.5pt);
\filldraw [blue] (8,-1) circle (1.5pt);

\draw [red] (0,-1.5) node[black][left] {$C_4$} -- (2,-1.5);
\filldraw [red] (0,-1.5) circle (1.5pt);
\filldraw [violet] (2,-1.5) circle (1.5pt);
\draw [blue](2,-1.5) -- (4,-1.5);
\filldraw [blue] (4,-1.5) circle (1.5pt);
\draw [red] (5,-1.5) -- (7,-1.5);
\filldraw [red] (5,-1.5) circle (1.5pt);
\filldraw [violet] (7,-1.5) circle (1.5pt);
\draw [blue](7,-1.5) -- (9,-1.5);
\filldraw [blue] (9,-1.5) circle (1.5pt);

\draw [red] (0,-2) node[black][left] {$C_5$}-- (2,-2);
\filldraw [red] (0,-2) circle (1.5pt);
\filldraw [violet] (2,-2) circle (1.5pt);
\draw [blue](2,-2) -- (4,-2);
\filldraw [blue] (4,-2) circle (1.5pt);
\draw [red] (4,-2) -- (6,-2);
\filldraw [red] (6,-2) circle (1.5pt);
\filldraw [violet] (4,-2) circle (1.5pt);
\draw [blue](7,-2) -- (9,-2);
\filldraw [blue] (9,-2) circle (1.5pt);
\filldraw [blue] (7,-2) circle (1.5pt);

\draw [red] (0,-2.5) node[black][left] {$C_6$}-- (2,-2.5);
\filldraw [red] (0,-2.5) circle (1.5pt);
\filldraw [violet] (2,-2.5) circle (1.5pt);
\draw [blue](2,-2.5) -- (4,-2.5);
\filldraw [blue] (4,-2.5) circle (1.5pt);
\draw [red] (4,-2.5) -- (6,-2.5);
\filldraw [violet] (6,-2.5) circle (1.5pt);
\filldraw [violet] (4,-2.5) circle (1.5pt);
\draw [blue](6,-2.5) -- (8,-2.5);
\filldraw [blue] (8,-2.5) circle (1.5pt);

\draw [red] (0,-3) node[black][left] {$C_7$} -- (2,-3);
\filldraw [red] (0,-3) circle (1.5pt);
\filldraw [violet] (2,-3) circle (1.5pt);
\draw [blue](2,-3) -- (4,-3);
\filldraw [blue] (4,-3) circle (1.5pt);
\draw [violet] (5,-3) -- (7,-3);
\filldraw [violet] (5,-3) circle (1.5pt);
\filldraw [violet] (7,-3) circle (1.5pt);

\draw [violet] (0,-3.5) node[black][left] {$C_8$} -- (2,-3.5);
\filldraw [violet] (0,-3.5) circle (1.5pt);
\filldraw [violet] (2,-3.5) circle (1.5pt);
\draw [violet] (3,-3.5) -- (5,-3.5);
\filldraw [violet] (3,-3.5) circle (1.5pt);
\filldraw [violet] (5,-3.5) circle (1.5pt);

\end{tikzpicture}

\caption{Kinds of pair of $2$-matchings in the sum of the second moment of $X$.}
\label{double_sum}
\end{figure}
The probabilities of that both $i$ and $j$ are crossing for each type of double $2$-matching are the following (with the obvious abuse of notation):
\begin{eqnarray*}
&&\mathbb{P}(C_1) = \frac{1}{9}, \quad \mathbb{P}(C_2)  = \frac{1}{9} , \quad \mathbb{P}(C_3) = \frac{2}{15}, \quad \mathbb{P}(C_4) = \frac{7}{60},\\
&&\mathbb{P}(C_5) = \frac{1}{10} , \quad \mathbb{P}(C_6) = \frac{1}{12}, \quad \mathbb{P}(C_7) = \frac{1}{6}, \quad \mathbb{P}(C_8) = \frac{1}{3}.
\end{eqnarray*}

Thus, we arrive to the following. 
\begin{lem} \label{lem:2ndmoment}
The second moment of $X$ is given by the formula,
\begin{equation}\label{EXX}
\mathbb{E}X^2= \frac{6}{9}m_4(G)+\frac{4}{5}m_3(G)+\frac{1}{3}m_2(G)+\frac{4}{9}S_2+ \frac{7}{15}S_4+\frac{1}{5}S_5+\frac{1}{6}S_6+ \frac{1}{3}S_7    
\end{equation}
where $S_i$ is the number of subgraphs of $G$ of  type $C_i$.
\end{lem}
Before proving our main result we will apply the above lemmas for a few examples.
\begin{exa}[Pairing]\label{pairing_example}
Consider $G$ to be a disjoint union of $n$ copies of $K_2$ graphs. The expectation is given by
$$\mathbb{E}X=\frac{1}{3} m_2(G)=\frac{n(n-1)}{6}.$$
For the variance, we only need to consider, $m_2(G)$, $m_4(G)$ and $m_3(G)$, since the other types of subgraphs are not present in $G$. The number of $r$-matchings is given by  $m_r(G)=\binom{n}{r}$, since any choice of $r$ different edges is an $r$-matching, then we obtain  $$\mathbb{E}X^2=\frac{6}{9}\binom{n}{4}+\frac{12}{15}\binom{n}{3}+\frac{1}{3}\binom{n}{2}=\frac{n^4}{36} - \frac{n^3}{30} + \frac{13 n^2}{180} - \frac{n}{15},$$
and thus the variance is 
$$\mathrm{Var}(X)=\mathbb{E}X^2-(\mathbb{E}X)^2=\frac{n(n-1)(n+3)}{45}.$$
\end{exa}
\begin{exa}[Path] \label{path_example}
Let $P_n$ be the path graph on $n$ vertices. In this case, the number of $r$-matchings is $m_r(P_n) = \binom{n-r}{r}$, so we obtain that 
\begin{eqnarray*}
&& m_4= \binom{n-4}{4},      \quad  m_3=\binom{n-3}{3},  \quad m_2=\binom{n-2}{2},  \quad S_2= 3 \binom{n-4}{3},  \\
&&  S_4 = \binom{n-4}{2}, \quad S_5 = 2\binom{n-4}{2}, \quad S_6 = n-4, \quad S_7 = 2 \binom{n-3}{2}.
\end{eqnarray*}
Then, 
$$\mathbb{E}X^2=\frac{n^4}{36}- \frac{23n^3}{90} +  \frac{35n^2}{36}-\frac{86 n}{45}- \frac{5}{3},$$
and thus the variance is given by
$$\mathrm{Var}(X) = \mathbb{E}X^2 -(\mathbb{E}X)^2 = \frac{n^3}{45}-\frac{n^2}{18}-\frac{11n}{45}+\frac{2}{3}.$$
\end{exa}
\begin{exa}[Cycle] \label{cycle_example}
Let $C_n$ be the cycle graph on $n$ vertices. The number of $r$-matchings is given by $m_r(C_n)=\frac{n}{r}\binom{n-r-1}{r-1}$, then
$$\mathbb{E}X=\frac{m_2}{3}=\frac{n(n-3)}{6}.$$
On the other hand, the number of subgraphs is given by
\begin{eqnarray*}
&& m_4= \frac{n}{4}\binom{n-5}{3},      \quad  m_3=\frac{n}{3}\binom{n-4}{2},  \quad m_2=\frac{n(n-3)}{2},  \quad S_2= n \binom{n-5}{2},  \\ &&
 S_4 =\frac{n(n-5)}{2}, \quad S_5 = n(n-5), \quad S_6 = n, \quad S_7 = n(n-4),
\end{eqnarray*}
from where the second moment is 
$$\mathbb{E}X^2=\frac{n^4}{36} - \frac{13 n^3}{90} + \frac{47 n^2}{180} - \frac{n}{3},$$
and thus the variance is given by
$$\mathrm{Var}(X) = \mathbb{E}X^2 -(\mathbb{E}X)^2 = \frac{n^3}{45}-\frac{ n^2}{90} - \frac{n}{3}.$$
\end{exa}
\begin{exa}[Triangles] \label{triangle_example}
Let $G$ be the disjoint union of $n$ copies of $K_3$.  In this case, the subgraphs of type $S_5$ and $S_6$ are not present in $G$. In order to obtain an $r$-matching for $r \geq 2$, we need to choose $r$ triangles and for each one we have 3 options to form the matching, so the number of $r$-matching is $m_r(G)=3^r \binom{n}{r}$. Similarly, the number of other types of subgraphs is given by
\begin{eqnarray*}
&& m_4 = 3^4 \binom{n}{4},    \quad m_3 = 3^3 \binom{n}{3}, \quad m_2 = 3^2 \binom{n}{2},   \\
&& S_2 = 3^4 \binom{n}{3},  \quad S_4  = 3^2 \binom{n}{2},  \quad S_7 = 2 \cdot 3^2 \binom{n}{2}.
\end{eqnarray*}
Then, the expectation and the second moment are 
$$\mathbb{E}X=\frac{1}{3} m_2(G)=\frac{3n(n-1)}{2}, \qquad \mathbb{E}X^2 =  \frac{9 n^4}{4} - \frac{39n^3}{10} + \frac{51n^2}{20} - \frac{9n}{10},$$
and thus the variance is 
$$\mathrm{Var}(X) = \mathbb{E}X^2 - (\mathbb{E}X)^2= \frac{3n^3}{5}+\frac{3n^2}{10}-\frac{9n}{10}.$$
\end{exa}

\section{Proof of the Main Theorem}

\subsection{Construction of Size Bias Transform}

Let $\pi$ be a fixed permutation and let $I = (e,f)$ be a random index chosen with probability $\mathbb{P}(I = i) = 1/ m_2(G)$ , which corresponds to a 2-matching (in this way $e,f$ are edges), and which is independent of $\pi$. For said fixed $\pi$, we have a fit $G_\pi$. We construct $\pi^s$ as follows
\begin{itemize}
    \item If $\pi$ is such that $G_\pi$ has a crossing at $I$, we define $\pi^s = \pi$.
    \item Otherwise, we perform the following process to obtain a permutation with a cross on $I$. Suppose $e = \{ u_1,u_2 \}$ and $f = \{ v_1,v_2 \}$, then under $\pi$ these edges are $\pi(e) = \{ \pi(u_1), \pi(u_2) \}$ and $\pi(f) = \{ \pi(v_1),\pi(v_2) \}$. Since they do not cross, without loss of generality we can assume that $\pi(u_1) < \pi(v_1) < \pi(v_2)<\pi(u_2)$ satisfies. Now, we choose a random vertex uniformly among the vertices of the edges of $I$. In case the vertex is $u_1$ or $v_1$, we leave these fixed and swap the positions between $\pi(v_2)$ and $\pi(u_2)$ and define $\pi^s$ as the resulting permutation. In case of choosing $u_2$ or $v_2$, we do the same process leaving these vertices fixed and exchanging $\pi(v_1)$ and $\pi(u_1)$. In this way we obtain a permutation $\pi^s$ such that it has a crossing at $I$
\end{itemize}
 Note that $\pi^s$ is a uniform random permutation conditional on the event that $\pi(u_i)$'s and $\pi (v_i)$'s are in alternating in the cyclic order. This in turn means that  $G_{\pi^s}$, is a uniform random embedding conditioned on the event that the event $I$ has a crossing.

In summary,  we obtain that 
$$X(G_{\pi^s}) = \sum_{i \in M_2(G)} \mathbb{I}_{\{ \pi^s(i) \text{ is a crossing in } G_{\pi^s} \}} = \sum_{i \in M_2(G)} Y_{\pi^s(i)},$$
satisfies $\{ Y_{\pi^s(j)} \}_{j \neq I} $ has the distribution of $\{ Y_{\pi(j)} \}_{j \neq I}$ conditional on $Y_{\pi(I)} =1$. Then, by Corollary \ref{corollary_sizebiasBernoulli}, we get that $X^s=X(G_{\pi^s}) $, has the size bias distribution of $X$.

Define $X^s = X(G_{\pi}^s)$ as the size bias transform of $X(G_\pi)$. By construction, $|X^s -X| \leq  2 \Delta (m-1)$. Indeed,  because in the worst case each of the edges incident to each vertex creates a new crossing.

\subsection{Bounding the Variance of $\mathbb{E}[X^s-X|X]$}
In order to use Theorem \ref{teoKolmo} we need a bound for the variance of a conditional expectation, which depends of $X$ and its size bias transform $X^s$. A bound for that variance is given in the next lemma, which is one of the main results in this paper.

\begin{lem}\label{lemmaVariance}
Let $G = (V,E)$ a graph with maximum degree $\Delta$, and let $X$ be the number of crossings of a random uniform embedding of $G$. Then 
\begin{equation}\label{bound_variance}
    \mathrm{Var}(\mathbb{E}[X^s- X|X]) \leq 4\Delta^2 (m-1)^2 \left( 1- \frac{6m_4(G)}{m_2(G)^2} + \frac{(\Delta-1)^2(m-4)}{2m_2(G)} \right),
\end{equation}
where $X^s$ is the size bias transform of $X$, $m$ is the number of edges of the graph $G$ and $m_r(G)$ is the number of $r$-matchings of the graph $G$.
\end{lem}
\begin{proof}
Note that 
\begin{align*}
\mathbb{E}[X^s - X | X] &=  \sum_{i \in M_2(G) } \mathbb{E}[X^s - X | X, I = i] \mathbb{P}(I = i) \\
&= \frac{1}{m_2(G)} \sum_{i \in M_2(G)} (X^{(i)} - X),
\end{align*}
where $X^{(i)}$ denote $X^s$ conditioned to have a crossing in the index $i$. This gives that 
\begin{equation}
    \mathrm{Var}(\mathbb{E}[X^s- X|X]) = \frac{1}{m_2(G)^2} \sum_{i,j \in M_2(G)} \mathrm{Cov}(X^{(i)}-X, X^{(j)}- X).
    \label{variance_covariancesum}
\end{equation}
We identify two kinds of terms in the summation of the covariance, when the indices satisfy $V(i) \cap V(j) \neq \emptyset$ or when they satisfy $V(i) \cap V(j) = \emptyset$, where $V(i)$ denote the set of vertices of the 2-matching $i$.

    \textbf{Case $V(i) \cap V(j) \neq \emptyset$}: In this case, we have that
    $$|\{i,j \in M_2(G) : V(i) \cap V(j) \neq \emptyset \}| = m_2(G)^2-\binom{4}{2} m_4(G) = m_2(G)^2- 6 m_4(G).$$ 
From the construction of the size bias transform, we have that $|X^{(i)}-X| \leq 2 \Delta (m-1)$. Indeed, if $X^s$ have a crossing in the matching $i$, there are two options for the matching $i$ in $X$, that is, it is a crossing or not. If $i$ is a crossing, then $X^{(i)} = X$,  because we don't need to crossing any edges. On the other hand, if $i$ is not a crossing of $X$, then to obtain $X^{(i)}$ it is necessary crossing the edges of the matching $i$, which can be generate at least a new crossing for each of the edges incidents to the four vertices of $i$. Then, an upper bound for the variance given by

    \begin{equation}\label{upperboundvariance}
        \mathrm{Var}(X^{(i)}-X)  \leq \mathbb{E}[(X^{(i)}-X)^2]\leq  4\Delta^2(m-1)^2.
    \end{equation}
    Then, the contribution in the sum \eqref{variance_covariancesum} of the 2-matchings such that $V(i) \cap V(j) \neq \emptyset$  is bounded by
    $$\frac{m_2(G)^2- 6m_4(G)}{m_2(G)^2} 4 \Delta^2(m-1)^2 = \left( 1-\frac{6m_4(G)}{m_2(G)^2} \right)4 \Delta^2(m-1)^2.$$

    \textbf{Case $V(i) \cap V(j) = \emptyset$}: In this case, we have that
    $$|\{i,j \in M_2(G) : V(i) \cap V(j) = \emptyset \}| = \binom{4}{2} m_4(G) = 6 m_4(G).$$
    Let $N(i)$ be the set of edges incidents to the vertices of the 2-matching $i$. We can divide the sum over the 2-matching with $V(i) \cap V(j) = \emptyset$. In the case of $N(i) \cap N(j) = \emptyset$, we obtain that the random variables $X^{(i)} -X$ and $X^{(i)} -X$ are independent. Indeed, from the construction $X^{(i)} -X$ only depends on the edges incident to the $2$-matching $i$, and similarly $X^{(j)} -X$ only depends on the edges incident to the $2$-matching $j$. Hence, in this case we obtain that $\mathrm{Cov}(X^{(i)}-X, X^{(j)}- X) =0$.
    
    So we are interested in the pairs of $2$- matchings such that $V(i)\cap V(j)=\emptyset$, but  $N(i) \cap N(j) \neq \emptyset$. An upper bound for the number of such pairs of $2$-matchings is  given by $m_2(G)(\Delta-1)^2 (m-4)/2 $.
    
    Indeed, in this case, there exists at least one edge between $V(i)$ and  $V(j)$.  So, to obtain such configuration one may proceed as follows. First, one chooses a $2$-matching, $i$, and one considers one of the $4$ vertices in $i$, say $v$, and looks for a neighbour of $v$, say $w$, which should be in the $2$-matching $j$. There are at most $\Delta-1$ choices for $w$. Now, to construct $j$, we need to find a neighbour of $w$ which is not $v$ for one of the edges forming $j$, giving at most $\Delta-1$ possibilities and another edge which cannot be in $i$, or contain $w$, giving at most $m-4$ possibilities.  Putting this together and considering double counting, we obtain the desired $m_2(G)(\Delta-1)^2(m-4)/2$. See Figure \ref{fig} for a diagram explaining this counting.
    
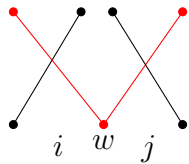
\begin{figure}[ht]
\centering 
\begin{tikzpicture}[scale = 1.5]
\draw (0,0)  -- (0.6,1) ;
\filldraw (0,0) circle (1 pt);
\filldraw (0.6,1) circle (1 pt);
\draw [red] (0,1) -- (0.8,0) ;
\filldraw [red] (0,1) circle (1 pt);
\filldraw [red] (1.5,1) circle (1 pt);
\filldraw [red] (0.8,0) node[black][below]{$w$} circle (1 pt);
\draw [red] (0.8,0) -- (1.5,1);
\draw (0.88,1)  -- (1.5,0) ;
\filldraw (0.88,1) circle (1 pt);
\filldraw (1.5,0) circle (1 pt); 
\filldraw (0.4,0) node[black][below]{$i$} circle (0 pt);
\filldraw (1.2,0) node[black][below]{$j$} circle (0 pt);
\end{tikzpicture}
\caption{A generic pair of $2$-matchings sharing a vertex.}
\label{fig}
\end{figure}

   Finally, using the upper bound for the variance given in \eqref{upperboundvariance}, we obtain
    \begin{align*}
        \frac{1}{m_2(G)^2} \sum_{\substack{i,j\in M_2(G)\\ V(i) \cap V(j) = \emptyset}} \mathrm{Cov}(X^{(i)}-X, X^{(j)}- X) &= \frac{1}{m_2(G)^2} \sum_{\substack{i,j\in M_2(G)\\N(i) \cap N(j) \neq \emptyset}} \mathrm{Cov}(X^{(i)}-X, X^{(j)}- X) \\
        &\leq \frac{1}{m_2(G)^2}\sum_{\substack{i,j\in M_2(G)\\N(i) \cap N(j) \neq \emptyset}} 4 \Delta^2 (m-1)^2 \\
        &\leq \frac{1}{m_2(G)}2\Delta^2(\Delta-1)^2 (m-1)^2 (m-4). 
    \end{align*}
    Thus, the contribution of the pairs of  $2$-matchings such that $V(i) \cap V(j) = \emptyset$ in the covariance sum \eqref{variance_covariancesum} is bounded by 
    $$\frac{2 \Delta^2(\Delta-1)^2 (m-1)^2 (m-4)}{m_2(G)}. $$

Therefore, 
$$\mathrm{Var}(\mathbb{E}[X^s- X|X]) \leq 4\Delta^2 (m-1)^2 \left( 1- \frac{6m_4(G)}{m_2(G)^2} + \frac{(\Delta-1)^2(m-4)}{2m_2(G)} \right).$$

\end{proof}

\subsection{Kolmogorov Distance}

Using the previous results, we are in position to apply Theorem \ref{teoKolmo}. Therefore, we can obtain a bound for the Kolmogorov distance of the (normalized) number of crossings and a standard Gaussian random variable.

\begin{teo}\label{teo:kolmo}
Let $G$ be a graph with maximum degree $\Delta$, and let $X$ be the number of crossings of a random uniform embedding of $G$. Let $\mu$ and $\sigma^2$ the mean and the variance of $X$. Then, with  $W = (X-\mu) /\sigma$,  
$$d_{Kol}(W,Z) \leq \frac{4 m_2(G) \Delta m}{3 \sigma^2}\left[\frac{6 \Delta m}{\sigma} + \sqrt{1- \frac{6m_4(G)}{m_2(G)^2} + \frac{(\Delta-1)^2m}{2m_2(G)} } \right],$$
where $m$ is the number of edges of $G$, $m_r(G)$ is the number of $r$-matchings of $G$, and $Z$ is a standard Gaussian random variable.
\end{teo}

\begin{proof}
By Lemma \ref{lem:mean} we have that $\mu =m_2(G) /3$, also by Lemma \ref{lemmaVariance}, $\Psi$ defined in \eqref{Psi} is bounded as follows,
$$\Psi \leq 2\Delta m \sqrt{  1- \frac{6m_4(G)}{m_2(G)^2} + \frac{(\Delta-1)^2 m}{2m_2(G)} }.$$
Then, using Theorem \ref{teoKolmo}, and the fact that $|X^s -X| \leq A =   2\Delta m$, we obtain
\begin{align*}
    d_{Kol}(W,Z) &\leq  \frac{6 \mu A^2}{\sigma^3} +\frac{2 \mu \Psi}{\sigma^2} \\
    &\leq \frac{8\Delta^2 m_2(G) m^2}{\sigma^3} + \frac{4 m_2(G) \Delta m}{3\sigma^2}  \sqrt{  1- \frac{6m_4(G)}{m_2(G)^2}+ \frac{(\Delta-1)^2m}{2m_2(G)}}\\
    &= \frac{4m_2(G)\Delta m}{3 \sigma^2} \left( \frac{6 \Delta m}{\sigma}  + \sqrt{  1- \frac{6m_4(G)}{m_2(G)^2}+ \frac{(\Delta-1)^2m}{2m_2(G)}}\right)
\end{align*}

\end{proof}

\section{Some Examples}

In this section we provide various examples for which Theorem \ref{teo:kolmo} can be applied directly. To show its easy applicability, we give explicit bounds on the quantities appearing in \eqref{kolmogorov_distance}.

\subsection{Pairing}

Let $G$ be a pairing or matching graph on $2n$ vertices, that is, a disjoint union of $n$ copies of $K_2$, as in Example \ref{pairing_example}. In particular, $m=n$,  $m_2(G)=\binom{n}{2}$ and $m_4(G)=\binom{n}{4}$ and the variance is given by $$\sigma^2=\frac{n(n-1)(n+3)}{45}$$ which is bigger that $n^3/45$ for $n>3$.
On the other hand, since $\Delta=1$, we see that, for $n>3$,
$$1-\frac{6 m_4(G)}{m_2(G)^2}+\frac{m(\Delta-1)^2}{2m_2(G)}=1-\frac{6 m_4(G)}{m_2(G)^2} =1- \frac{6 \binom{n}{4}}{\binom{n}{2}^2} = \frac{4n-6}{n^2-n} =\frac{4}{n} -\frac{2}{n^2-n}<\frac{4}{n}.$$
Thus, 
$$d_{Kol}(W,Z) \leq \frac{4\cdot 45 n^3}{3\cdot 2 n^{3}}\left(\frac{6\cdot \sqrt{45} n}{n^{3/2}} + \frac{2}{\sqrt{n}} \right)\leq \frac{1268}{\sqrt{n} }.$$

\subsection{Path Graph} 

Let $P_n$ be the path graph on $n$ vertices. From Example \ref{path_example}, $m_2(P_n)=\binom{n-2}{2}$ and $m_4(P_n)=\binom{n-4}{4}$, and we obtain that
$$1-\frac{6 m_4(P_n)}{m_2(P_n)^2} = 1-\frac{6 \binom{n-4}{4}}{\binom{n-2}{2}^2}  =  \frac{2(6n^3-71n^2+289n-402)}{n^4- 10 n^3 + 37 n^2- 60 n+36}=\frac{12}{n}+o(n^{-1})$$
On the other hand, $\Delta=2$, and then, one easily sees that,
$$4 m_2(P_n) \Delta m\leq 4n^3,\qquad 6\Delta m\leq 12n,\quad \text{ and }\quad \frac{m(\Delta-1)^2}{2m_2(P_n)}=\frac{1}{(n-2)}.$$
Finally, since the variance is given by $\sigma^2 =  n^3/45-n^2/18 - 11 n/45 -2/3>n^3/60$ , for $n\geq 14$, we get 
$$d_{Kol}(W,Z) \leq \frac{4 \cdot 60 n^3}{3 n^{3}}\left(\frac{12\cdot \sqrt{60} n}{n^{3/2}} + \frac{4}{\sqrt{n}} \right)\leq \frac{7757}{\sqrt{n} }.$$

\subsection{Cycle Graph}

Let $C_n$ be the cycle graph on $n$ vertices. In this case $m=n$ and $\Delta = 2 $, and from Example \ref{cycle_example}, $m_2(C_n)=\frac{n(n-3)}{2}$ and $m_4(C_n)=\frac{n}{4}\binom{n-5}{3}$, then
$$1-\frac{6 m_4(C_n)}{m_2(C_n)^2} +\frac{(\Delta-1 )^2 m}{2m_2(C_n)}= 1- \frac{6\frac{n}{4} \binom{n-5}{3}}{\left(\frac{n(n-3)}{2} \right)^2} + \frac{n}{n(n-3)} =  \frac{13n^2-101n+210}{n(n-3)^2} \leq \frac{13}{n},\text{ for } n\geq 5.$$
Also, $4m_2(C_n) \Delta m \leq 4n^3$, and $6 \Delta m = 12 n$. Since the variance is  $\sigma^2 =  n^3/45 - n^2/90 -n/3 >n^3/50, $ for $n\geq 15$, we obtain 
$$d_{Kol}(W,Z) \leq \frac{4 \cdot 50 n^3}{3 n^{3}}\left(\frac{12 \sqrt{50} n}{n^{3/2}} + \frac{\sqrt{13}}{\sqrt{n}} \right)\leq \frac{{5898}}{\sqrt{n} }.$$

\subsection{Disjoint Union of Triangles} Consider $n$ copies of $K_3$ and let $G$ be the disjoint union of them. From Example \ref{triangle_example} we have that $G$ is a graph with $3n$ vertices, $m=3n$ edges, maximum degree $\Delta =2$, $m_2(G)=3^2\binom{n-2}{2}$ and $m_4(G)=3^4\binom{n-4}{4}$, then
$$1-\frac{6 m_4(G)}{m_2(G)^2} +\frac{(\Delta-1 )^2 m}{2m_2(G)}= 1- \frac{6 \left(3^4\binom{n}{4}\right)}{\left( 3^2\binom{n}{2}\right)^2}+\frac{3n}{2\cdot 3^2 \binom{n}{2}} = \frac{13n-18}{3n(n-1)} \leq \frac{13}{3n}.$$
On the other hand, we can obtain that, $4m_2(G) \Delta m \leq 108 n^3$ and $6 \Delta m = 36n$. Finally, since the variance is $\sigma^2 = 3n^3/5+3n^2/10-9n/10 > 3n^3/5,$ for $n \geq 3$, then we get
$$d_{Kol}(W,Z) \leq \frac{108 \cdot 5 n^3}{9 n^{3}}\left(\frac{36 \sqrt{5} n}{\sqrt{3} n^{3/2}} + \frac{\sqrt{13}}{\sqrt{2n}} \right)\leq \frac{2942}{\sqrt{n} }.$$

\section{Another Possible Limit}

The following shows that not every sequence of graphs satisfies a central limit theorem for the number of crossings, even if the variance is not always $0$ and that having $m_2$ going to infinity is not enough.  Moreover, it shows that we can have another type of limit for the number of crossings.

Consider the graph $G_n$ which consists of a star graph with $n-1$ vertices for which an edge is attached at one of the leaves, as in Figure \ref{kite}a. 

Note that in this case $m_2=n-3$ and $m_4 = 0$ and the only other term appearing in \eqref{EXX} is  $S_7$, which for this graph equals $\binom{n-3}{2}$ . This implies that $$\mathbb{E}(X)=\frac{n-3}{3}, \quad \mathbb{E}(X^2)=\frac{(n-2)(n-3)}{6},$$
from where $\sigma^2=n(n-3)/18$,  $1-6m_4/m_2^2=1$ and 
$$\frac{(\Delta-1)^2m}{2m_2} = \frac{(n-2)^2(n-1)}{2(n-3)} \approx \frac{n^2}{2}.$$
Thus the right hand of \eqref{kolmogorov_distance} does not approximate $0$ as $n\to \infty$.

One can calculate explicitly the probability of having $k$ crossings. Indeed, lets us denote by $v_0$ is the center and by $v_n$, the tail (the only vertex at distance 2 from $v_0$) and by $v_{n-1}$ the vertex which has $v_0$ and $v_n$. The number of crossings in an embedding of $G_n$ depends only on the position of this three vertices. More precisely, there will be exactly $k$ crossing if the following two conditions are satisfied (see Figure 5b for an example):
\begin{enumerate}
\item There are exactly $k$ and $n-2-k$ vertices in the two arcs that remain when removing $v_n$ and $v_{n-1}$.
\item $v_0$ is in the arc with $n-2-k$ vertices.
\end{enumerate}

\begin{figure}[ht]
\centering 
\begin{tikzpicture}
[mystyle/.style={scale=0.8, draw,shape=circle,fill=black}]
\def\ngon{10}
\node[regular polygon,regular polygon sides=\ngon,minimum size=2.5cm] (p) {};
\foreach\x in {1,...,\ngon}{\node[mystyle] (p\x) at (p.corner \x){};}
\node[mystyle] (p0) at (0,0) {};
\foreach\x in {1,...,\ngon}
{
 \draw[thick] (p0) -- (p\x);
}
\node[mystyle] (p00) at (-3.2,0) {};
{
 \draw[thick] (p00) -- (p4);
}

\node [label=below:a) $G_{12}$] (*) at (0,-2) {};
 \end{tikzpicture}
\qquad \qquad
\begin{tikzpicture}
[mystyle/.style={scale=0.8, draw,shape=circle,fill=black}]
\def\ngon{10}
\node[regular polygon,regular polygon sides=\ngon,minimum size=3cm] (p) {};
\foreach\x in {1,...,\ngon}{\node[mystyle] (p\x) at (p.corner \x){};}

\foreach\x in {5,...,\ngon}
{
 \draw[thick] (p1) -- (p\x);
}

\foreach\x in {1,...,3}
{
 \draw[thick] (p1) -- (p\x);
}

\draw[thick] (p4) -- (p8);

\node [label=below: $v_{0}$] (*) at (1,2) {};
\node [label=below: $v_{n-1}$] (*) at (1.8,-.4) {};
\node [label=below: $v_{n}$] (*) at (-1.9,0) {};

\node [label=below:b) $G_{10}$] (*) at (0,-2) {};
 \end{tikzpicture}

\caption{a) $G_{12}$ drawn without crossing (left). b) $G_{10}$ drawn in convex position with $3$ crossings formed by the edge $(v_{n-1},v_{n})$ and the edges $(v_0,v_i)$ where $v_i$ are between $v_{n-1}$ and $v_{n}$ in the cyclic order.} \label{kite}

\end{figure}
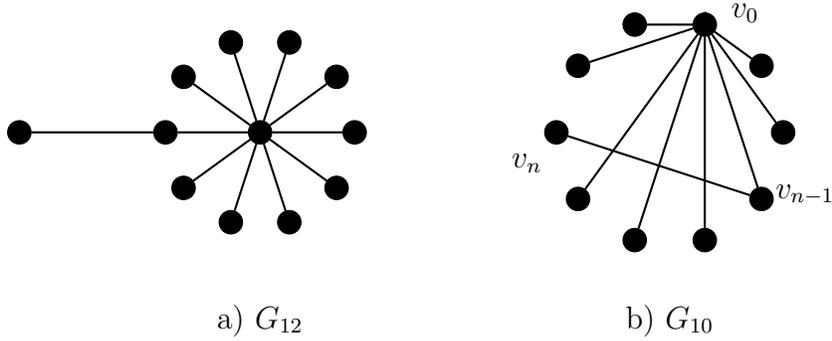

By a simple counting argument, since all permutations have the same probability of occurrence, one sees that such conditions are be satisfied with probability
$$\mathbb{P}(X_n=k)=\frac{2(n-2-k)}{(n-1)(n-2)}, \qquad k=0,\ldots,n-2.$$
Finally,
dividing by $n$, the random variable $Y_n=X_n/n$  satisfies that
$$\mathbb{P}\left(Y_n=\frac{k}{n}\right)=\frac{2(n-2-k)}{(n-1)(n-2)}\approx\frac{2}{n}\left(1-\frac{k}{n}\right), \qquad k=0,\ldots,n-2,$$
which implies that $Y_n\to Y$, weakly, where $Y$ is a random variable supported on $(0,1)$ with density $f_Y(x)=2(1-x).$

\subsection*{Acknowledgement}

We would like to thank Professor Goldstein for pointing out the paper \cite{paguyo2021convergence} and for various communications during the preparation of this paper. OA would like to thank Clemens Huemer for initial discussions on the topic that led to work in this problem.

Santiago Arenas-Velilla was supported by a scholarship from CONACYT.

Octavio Arizmendi was supported by CONACYT Grant CB-2017-2018-A1-S-9764.
\\
{\begin{minipage}[l]{0.3\textwidth} \includegraphics[trim=10cm 6cm 10cm 5cm,clip,scale=0.15]{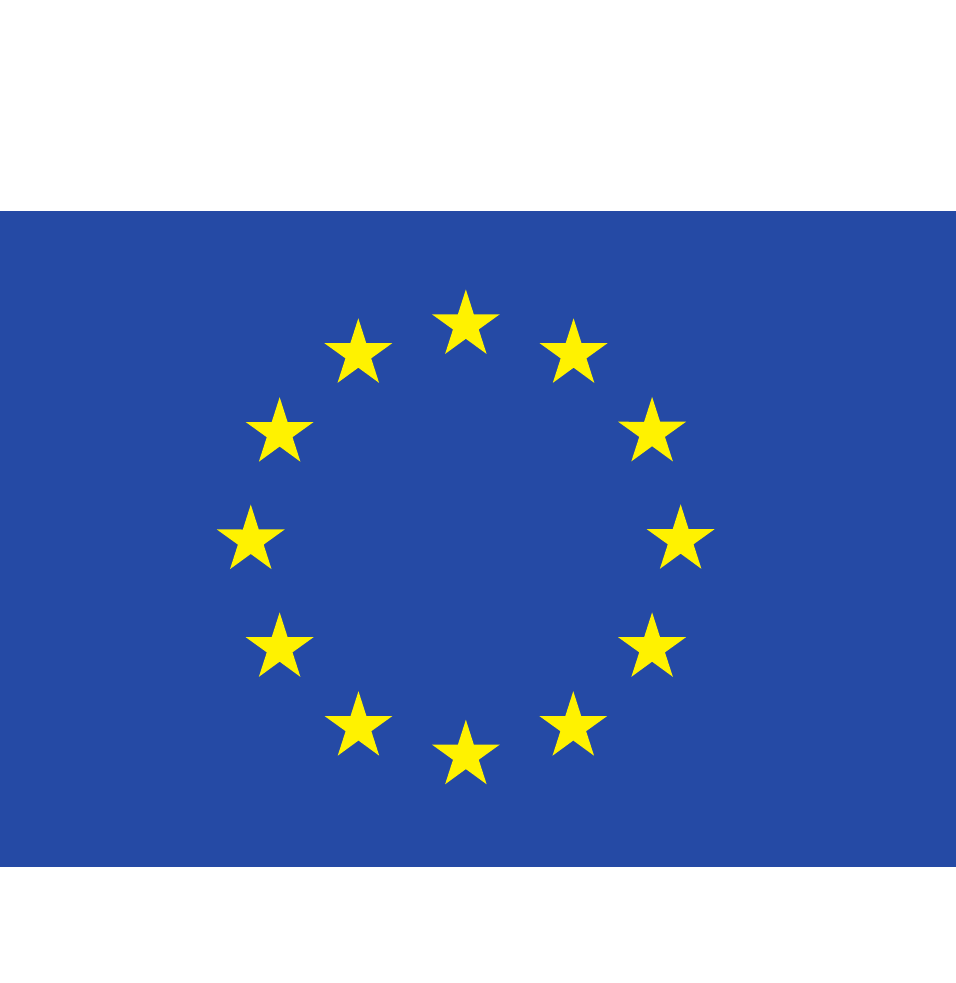} \end{minipage}
 \hspace{-3.2cm} \begin{minipage}[l][1cm]{0.82\textwidth}
 	  This project has received funding from the European Union's Horizon 2020 research and innovation programme under the Marie Sk\l{}odowska-Curie grant agreement No 734922.
		 	\end{minipage}}

\bibliographystyle{plain}
\bibliography{references}

\end{document}